\documentclass{amsart}
\usepackage{paralist}
\usepackage{amsmath}
\usepackage{amssymb}
\usepackage{amsthm}

 \newtheorem{theorem}{Theorem}[section]
 \newtheorem{proposition}[theorem]{Proposition}
 \newtheorem{lemma}[theorem]{Lemma}
 \newtheorem{corollary}[theorem]{Corollary}
 \theoremstyle{definition}
 
 \newtheorem{definition}[theorem]{Definition}
 \newtheorem{examples}[theorem]{Examples}

\begin{document}

    \title{Module decompositions using pairwise comaximal ideals}

    \author{Gary F. Birkenmeier}
    \address{Department of Mathematics, University of Louisiana at Lafayette, Lafayette, LA 70504 USA}
    \email{gfb1127@louisiana.edu}

    \author{C. Edward Ryan}
    \address{Department of Mathematics, University of Louisiana at Lafayette, Lafayette, LA 70504 USA}
    \email{cxr2665@louisiana.edu}



\begin{abstract}
In this paper we show that for a given set of pairwise comaximal ideals $\{X_i\}_{i\in I}$ in a ring $R$ with unity and any right $R$-module $M$ with generating set $Y$ and $C(X_i)=\sum\limits_{k\in\mathbb{N}}\underline{\ell}_M(X_i^{k})$, $M=\oplus_{i\in I}C(X_i)$ if and only if for every $y\in Y$ there exists a nonempty finite subset $J\subseteq I$ and positive integers $k_j$ such that $\bigcap\limits_{j\in J}X_i^{k_j}\subseteq\underline{r}_R(yR)$. We investigate this decomposition for a general class of modules. Our main theorem can be applied to a large class of rings including semilocal rings $R$ with the Jacobson radical of $R$ equal to the prime radical of $R$, left (or right) perfect rings, piecewise prime rings, and rings with ACC on ideals and satisfying the right AR property on ideals. This decomposition generalizes the decomposition of a torsion abelian group into a direct sum of its p-components. We also develop a torsion theory associated with sets of pairwise comaximal ideals.
\end{abstract}

\maketitle

\section{Introduction}

Throughout this paper $R$ denotes a ring, not necessarily commutative, with identity, and $M$ denotes a unital right $R$-module.

Recall if $M$ is a torsion abelian group, then $M=\oplus C(P_i)$, where $C(P_i)$ is the $p$-component of $M$. Also, if $R$ is semisimple Artinian, then $M=\oplus\underline{\ell}_M(P_i)$, where $P_i$ is a maximal ideal of $R$ and $\underline{\ell}_M(P_i)=\{m\in M\mid mP_i=0\}$ is the homogeneous component of $P_i$ in $M$. It is natural to investigate a general decomposition theory that includes the aforementioned decomposition results as special cases. In Section~\ref{sec:main} we provide such a result.

In Section~\ref{sec:main} we state a decomposition theorem which provides a decomposition of a module as a direct sum of fully invariant submodules. Our main result, Theorem~\ref{thm:infinite}, decomposes a right $R$-module $M$ to a direct sum in terms of annihilator submodules using a set of pairwise comaximal ideals of $R$.

Section~\ref{sec:nilary} extends the primary decomposition of a finitely generated torsion module over a Dedekind domain to certain kinds of noncommutative rings (cf. Theorem~\ref{thm:nilarydecomposition}).

We develop a preradical $\gamma$ and its radical closure $\bar{\gamma}$ based on Theorem~\ref{thm:infinite} in Section~\ref{sec:torsion}. Our main goal in this section is to obtain a decomposition of a module into a direct sum of a torsion module and a torsion-free module (cf. Proposition~\ref{prop:splitting}) under $\gamma$.

We write $K\subseteq M$ and $K\leq M$ to denote subsets and submodules of $M$, respectively. We say that a submodule $N\leq M$ is \textit{essential} in $M$, denoted $N\mathop{\leq^{\rm ess}\!}M$, if $N\cap K\neq0$ for any nonzero submodule $K\leq M$. A submodule $N\leq M$ is \textit{fully invariant} in $M$, denoted $N\trianglelefteq M$, if and only if $f(N)\subseteq N$ for every $f\in\rm{End}_R(M)$, where $\rm{End}_R(M)=\{h:M\longrightarrow M\mid h\text{ is an }R\text{-homomorphism}\}$. If $X\subseteq R$, then the \textit{left annihilator} of $X$ in $M$ is $\underline{\ell}_M(X)=\{m\in M\mid mx=0\text{ for all }x\in X\}$. If $N\subseteq M$, then the \textit{right annihilator} of $N$ in $R$ is $\underline{r}_R(N)=\{r\in R\mid nr=0\text{ for all }n\in N\}$. The \textit{singular submodule} of $M$, denoted by $Z(M)$, is $Z(M)=\{m\in M\mid\underline{r}_R(m)\mathop{\leq^{\rm ess}\!}R\}$.

As in \cite[p.123]{Row-1991}, we say a nonempty set of ideals $\{X_i\}_{i\in I}$ of $R$ is \textit{pairwise comaximal} if and only if $X_i+X_j=R$ for all $i,j\in I$ with $i\neq j$.

\section{Decomposition Theorem}\label{sec:main}

In this section we develop a decomposition theorem, Theorem~\ref{thm:infinite}, which provides a basic decomposition of a module into a direct sum of fully invariant annihilator submodules. This result also generalizes the well known result that a torsion abelian group is a direct sum of its p-components.

We make repeated and implicit use of the following lemma, which lists some properties of a finite collection of pairwise comaximal ideals of a ring. These properties are well known.

\begin{lemma}\label{lem:props}
\begin{enumerate}
\item If $\{X_i\}_{i=1}^n$, $n>1$, is a set of pairwise comaximal ideals of $R$, then $\sum\limits_{i=1}^n\bigl(\bigcap\limits_{j\neq i}X_j\bigr)=R=\sum\limits_{i=1}^n\bigl(\prod\limits_{j\neq i}X_j\bigr)$.
\item If $\{X_i\}_{i=1}^n$ is a set of pairwise comaximal ideals of $R$, then $X_i+\bigcap\limits_{j\neq i}X_j=R$. In particular, if each $X_i\neq R$, then $\bigcap\limits_{j\neq i}X_j\neq0$ for all $i$.
\item If $\{Y_i\}_{i=1}^n$ is a set of ideals of $R$ such that $\sum\limits_{i=1}^nY_i=R$, and we define $X_i=\sum\limits_{i\neq j}Y_j$, then $\bigcap\limits_{i=1}^nX_i=\sum\limits_{\sigma\in S_n}X_{\sigma(1)}X_{\sigma(2)}\dotsb X_{\sigma(n)}$, where $S_n$ is the set of permutations on $n$ letters.
\item If $X,Y\trianglelefteq R$ are such that $X+Y=R$, then $X^m+Y^n=R$ for each pair of positive integers $m,n$.
\end{enumerate}
\end{lemma}

\begin{definition}
Let $X$ be an ideal of $R$, and $M$ be an $R$-module. We define the \textit{component} of $X$ in $M$ to be $C(X)=\sum\limits_{k\in\mathbb{N}}\underline{\ell}_M(X^k)$.
\end{definition}

\begin{theorem}\label{thm:infinite}
Let $\mathcal{X}=\{X_i\}_{i\in I}$ be a nonempty collection of pairwise comaximal ideals of $R$, and let $Y$ be a generating set for $M$. Then:
\begin{enumerate}
    \item for each $y\in Y$, there exists a finite subset $J\subseteq I$ such that $\bigcap\limits_{j\in J}X_j\subseteq\underline{r}_R(y)$ if and only if $M=\oplus_{i\in I}\underline{\ell}_M(X_i)$;
    \item for each $y\in Y$, there exists a finite subset $J\subseteq I$ and positive integers $k_j$, $j\in J$, such that $\bigcap\limits_{j\in J}X_j^{k_j}\subseteq\underline{r}_R(y)$ if and only if $M=\oplus_{i\in I}C(X_i)$.
\end{enumerate}
\end{theorem}
\begin{proof}
(1) The proof of (1) is similar to that of (2).

(2) ($\Rightarrow$) Let $0\neq y\in Y$. Then $\underline{r}_R(yR)\supseteq\bigcap\limits_{i=1}^nX_i^{k_i}$ for some $X_i\in\mathcal{X}$, $k_i\geq1$, since $\bigcap\limits_{i=1}^nX_i^{k_i}$ is an ideal of $R$. We can write $y$ as $y=yx_{2,1}x_{3,1}\dotsb x_{n,1}+yx_{1,2}x_{3,2}\dotsb x_{n,2}+\dotsb+yx_{1,n}x_{2,n}\dotsb x_{n-1,n}$, where $x_{i,j}\in X_i^{k_i}$ and $x_{2,1}x_{3,1}\dotsb x_{n,1}+\dotsb+x_{1,n}x_{2,n}\dotsb x_{n-1,n}=1$, since $\{X_i^{k_i}\}_{i=1}^n$ is pairwise comaximal. Note that the $i^{th}$ term of the sum is an element of $\underline{\ell}_M(X_i^{k_i})$. Thus $y\in\sum\limits_{X\in\mathcal{X}}C(X)$, so $M=\sum\limits_{i\in I}C(X_i)$.

To show that $\{C(X)\mid X\in\mathcal{X}\}$ is an independent set, take $X\in\mathcal{X}$ and a finite subset $\{X_i\}_{i=1}^n$ of $\mathcal{X}$ not containing $X$. Notice that $\underline{r}_R\bigl(\underline{\ell}_M(X^k)\cap\sum\limits_{i=1}^n\underline{\ell}_M(X_i^{k_i})\bigr)\supseteq\underline{r}_R\bigl(\underline{\ell}_M(X^k)\bigr)+\bigcap\limits_{i=1}^n\underline{r}_R\underline{\ell}_M(X_i^{k_i})\supseteq X^k+\bigcap\limits_{i=1}^nX_i^{k_i}=R$ for any $k,k_i\geq1$. Thus $\underline{\ell}_M(X^k)\cap\sum\limits_{i=1}^n\underline{\ell}_M(X_i^{k_i})=0$ for $k,k_i\geq1$, so $C(X)\cap\sum\limits_{X'\neq X}C(X')=0$. Therefore $C(X)\mid X\in\mathcal{X}\}$ is an independent set.

($\Leftarrow$) Let $y\in Y$. Then $y=y_1+y_2+\cdots+y_n$ for some $y_i\in C(X_i)$ with $X_i\in\mathcal{X}$, $1\leq i\leq n$. For each $i$, there is a minimum power $k_i$ for which $y_i\in\underline{\ell}_M(X_i^{k_i})$. Thus we have $\underline{r}_R(yR)\supseteq\underline{r}_R\bigl(\oplus_{i=1}^n\underline{\ell}_M(X_i^{k_i})\bigr)=\bigcap\limits_{i=1}^n\underline{r}_R\underline{\ell}_M(X_i^{k_i})\supseteq\bigcap\limits_{i=1}^nX_i^{k_i}$.
\end{proof}

Note that in the above result, each $\underline{\ell}_M(X_i)\trianglelefteq M$, and each $C(X_i)\trianglelefteq M$. Hence if $M=\oplus_{i\in I}\underline{\ell}_M(X_i)$ or $M=\oplus_{i\in I}C(X_i)$, then $\rm{End}(M_R)=\prod\rm{End}\bigl(\underline{\ell}_M(X_i)_R\bigr)$ or $\rm{End}(M_R)=\prod\rm{End}\bigl(C(X_i)_R\bigr)$, respectively.

For another immediate example illustrating Theorem~\ref{thm:infinite}, let $\{X_i\}_{i\in\mathbb{N}}$ be a set of pairwise comaximal ideals of $R$ and $M=\oplus_{i\in\mathbb{N}}R/X_i^{i}$. For any $m\in M$, there exists a finite nonempty subset $J\subseteq I$ such that $\bigcap\limits_{j\in J}X_j^j\subseteq\underline{r}_R(mR)$. By Theorem~\ref{thm:infinite}, $M=\oplus_{i\in I}C(X_i)$. Note that $C(X_i)=R/X_i^i$ for each $i\in\mathbb{N}$.

\begin{corollary}\label{thm:main}
Let $\{X_i\}_{i=1}^n$, $n\geq2$, be a set of pairwise comaximal ideals of $R$, and let $A=\bigcap\limits_{i=1}^nX_i$. Then:
\begin{enumerate}
    \item $A\subseteq\underline{r}_R(M)$ if and only if $M=\underline{\ell}_M(X_1)\oplus\dotsb\oplus\underline{\ell}_M(X_n)$;
    \item $\bar{M}=M/MA=\oplus_{i=1}^n\underline{\ell}_{\bar{M}}(X_i)$.
\end{enumerate}
\end{corollary}

Note that Corollary~\ref{thm:main} can also be proven using the Chinese Remainder Theorem (cf., e.g., \cite[p.131]{Hun-1974}).

Recall that a prime Goldie ring $R$ in which each nonzero ideal of $R$ is invertible is called an \textit{Asano order} (or an \textit{Asano prime ring} \cite[pp.146--150]{MR-2001}). For example, a Dedekind domain is an Asano order. The next result is an application of Theorem~\ref{thm:infinite} to rings with sets of commuting pairwise comaximal ideals. Note that in an Asano order, multiplication of maximal ideals is commutative, and every nonzero ideal is a unique product of maximal ideals.

\begin{corollary}
Suppose that $\{X_i\}_{i\in I}$ is a set of commuting pairwise comaximal ideals of $R$ (i.e., $X_iX_j=X_jX_i$ for all $i,j\in I$). Let $M$ be a nonzero $R$-module and $Y$ a generating set of $M$. If, for every $y\in Y$, there exists a finite subset $J\subseteq I$ and positive integers $k_j$ such that $\prod\limits_{j\in J}X_j^{k_j}\subseteq\underline{r}_R(yR)$, then $M=\oplus_{i\in I}C(X_i)$.
\end{corollary}
\begin{proof}
The corollary follows from Lemma~\ref{lem:props} and Theorem~\ref{thm:infinite}.
\end{proof}

\begin{corollary}\label{cor:Asano}
Let $R$ be an Asano prime ring and $\{X_i\}_{i\in I}$ be the set of maximal ideals of $R$, and let $Y$ be a generating set of $M$. If $\underline{r}_R(yR)\neq0$ for all $y\in Y$, then $M=\oplus_{i\in I}C(X_i)$.
\end{corollary}

Note that Corollary~\ref{cor:Asano} generalizes the well-known theorem that every torsion abelian group is the direct sum of its $p$-components. That is, in the case where $R=\mathbb{Z}$ and $M$ is an abelian torsion group, then Theorem~\ref{thm:infinite} yields the decomposition of $M$ into its $p$-components.

A natural question to ask is: ``Under what conditions can we guarantee that each annihilator direct summand of the decomposition afforded by Theorem~\ref{thm:infinite} or Corollary~\ref{thm:main} is nonzero?'' For example, take $R=\mathbb{Z}$, and let $M=\mathbb{Z}_2\oplus\mathbb{Z}_3$. Consider $\{2\mathbb{Z},3\mathbb{Z},5\mathbb{Z}\}$. Then the conditions of Corollary~\ref{thm:main}(1) are satisfied, so $M=\underline{\ell}_M(2\mathbb{Z})\oplus\underline{\ell}_M(3\mathbb{Z})\oplus\underline{\ell}_M(5\mathbb{Z})$. But $\underline{\ell}_M(5\mathbb{Z})=0$.

The next theorem gives a set of conditions which ensures the non-triviality of the direct summands.

\begin{theorem}\label{thm:minimal}
Let $\{X_i\}_{i\in I}$ be a family of pairwise comaximal ideals of $R$. Then $M=\oplus_{i\in I}\underline{\ell}_M(X_i)$ and each $\underline{\ell}_M(X_i)\neq0$ if and only if
\begin{enumerate}
    \item for every $m\in M$, there exists a nonempty finite subset $J\subseteq I$ such that $\bigcap\limits_{i\in J}X_i\subseteq\underline{r}_R(mR)$; and
    \item for every $X_j$, $j\in I$, there exists $m\in M$ such that for some nonempty finite subset $J\subseteq I$ with $j\in J$, $\bigcap\limits_{i\in J}X_i\subseteq\underline{r}_R(mR)$ and $\bigcap\limits_{i\in J-\{j\}}X_i\nsubseteq\underline{r}_R(mR)$.
\end{enumerate}
\end{theorem}

\begin{proof}
$\Rightarrow$) Suppose that $M=\oplus_{i\in I}\underline{\ell}_M(X_i)$ and each $\underline{\ell}_M(X_i)\neq0$. By Theorem~\ref{thm:infinite}, for every $m\in M$, there exists a nonempty finite subset $J\subseteq I$ such that $\bigcap\limits_{i\in J}X_i\subseteq\underline{r}_R(mR)$. Let $j\in I$. Since $\underline{\ell}_M(X_j)\neq0$, we can find a nonzero $m\in\underline{\ell}_M(X_j)$. Then $\underline{r}_R(mR)\supseteq\underline{r}_R\underline{\ell}_M(X_j)\supseteq X_j$.

($\Leftarrow$) By Theorem~\ref{thm:infinite}, $M=\oplus_{i\in I}\underline{\ell}_M(X_i)$. If $\underline{\ell}_M(X_j)=0$ for some $j\in I$, then for every $m\in M$ there exists a subset $J\subseteq I$ with $j\notin J$ such that $\underline{r}_R(mR)\supseteq\bigcap\limits_{i\in J}X_i$, which is a contradiction.
\end{proof}

The following examples illustrate some of the results of this section.

\begin{examples}\label{section3examples}
(1) Let $X,Y,Z\trianglelefteq R$ such that $X+Y+Z=R$.
\\Let $T=\begin{bmatrix} R& 0& 0\\0& R& Z\\0& 0& R\end{bmatrix}$, and let
$M=\begin{bmatrix} \frac{R}{X+Z}& \frac{R}{Y+Z}& \frac{R}{X+Y}\\ \frac{R\vphantom{3^{3^3}}}{X+Y\vphantom{3_{3_3}}}& \frac{R}{Y+Z}&
\frac{R}{X+Y}\\ \frac{R}{Y+Z}& \frac{R}{X+Z}& \frac{R}{Y+Z}\end{bmatrix}$, where $\frac{R}{X+Y}$ denotes the factor ring of $R$ by $X+Y$. Then
$$\underline{r}_T(M)=\begin{bmatrix} (X+Y)\cap(X+Z)\cap(Y+Z)& 0& 0\\0& (X+Z)\cap(Y+Z)& Z\\0& 0& (X+Y)\cap(Y+Z)\end{bmatrix}.$$
Let $P_1=\begin{bmatrix} X+Y& 0& 0\\0& X+Z& Z\\0& 0& X+Y\end{bmatrix}$,
$P_2=\begin{bmatrix} X+Z& 0& 0\\0& Y+Z& 0\\0& 0& Y+Z\end{bmatrix}$, and
\\$P_3=\begin{bmatrix} Y+Z& 0& 0\\0& R& Z\\0& 0& R\end{bmatrix}$.
Then $P_i+P_j=T$ for all $i\neq j$, and
$P_1\cap P_2\cap P_3\subseteq\underline{r}_T(M)$; so by Corollary~\ref{thm:main},
\begin{align*}
M & =\underline{\ell}_M(P_1)\oplus\underline{\ell}_M(P_2)\oplus\underline{\ell}_M(P_3) \\
& =\begin{bmatrix} 0& 0& \frac{R}{X+Y}\\ \frac{R\vphantom{3^{3^3}}}{X+Y\vphantom{3_{3_3}}}& 0& \frac{R}{X+Y}\\0& \frac{R}{X+Z}& 0
\end{bmatrix}
\oplus\begin{bmatrix} \frac{R}{X+Z}& \frac{R}{Y+Z}& 0\\0& \frac{R\vphantom{3^{3^3}}}{Y+Z\vphantom{3_{3_3}}}& 0\\0& 0&
\frac{R}{Y+Z}\end{bmatrix}
\oplus\begin{bmatrix} 0& 0& 0\\0& 0& 0\\ \frac{R}{Y+Z}& 0& 0\end{bmatrix}.
\end{align*}

(2) Let $R$ be a Dedekind domain and $\{P_1,P_2\}$ a family of distinct nonzero prime ideals of $R$. Suppose that $M=R/P_1^2\oplus R/P_1P_2$. Then $X_1=P_1^2$ and $X_2=P_2$ are two ideals of $R$ such that
\begin{inparaenum}
    \item $P_1^2P_2=\underline{r}_R(M)=X_1\cap X_2$ and
    \item $X_1+X_2=R$.
\end{inparaenum}
Therefore, $M=\underline{\ell}_M(X_1)\oplus\underline{\ell}_M(X_2)$, where $\underline{\ell}_M(X_1)=R/P_1^2\oplus P_2/P_1P_2$ and $\underline{\ell}_M(X_2)=0\oplus P_1/P_1P_2$.
\end{examples}

\begin{lemma}\label{lem:semiqBaer}
Consider the following conditions on $R$:
\begin{enumerate}
\item $R/P(R)$ is quasi-Baer,
\item Every prime ideal of $R$ contains a unique minimal prime ideal, and
\item Every pair of distinct minimal prime ideals is comaximal.
\end{enumerate}
Then (2)$\Longleftrightarrow$(3). Moreover, if $R$ has only finitely many minimal prime ideals, then (1)$\Longleftrightarrow$(2).
\end{lemma}

\begin{proof}
Suppose every prime ideal contains a unique minimal prime ideal. Assume that there exist two minimal prime ideals $P_1, P_2$ of $R$ such that $P_1+P_2\subsetneq R$. Then there exists a maximal ideal $M$ such that $P_1+P_2\subseteq M$. Since $M$ is maximal, $M$ is a prime ideal and $P_1,P_2\subseteq M$, which is a contradiction.

Suppose that every pair of distinct minimal prime ideals is comaximal. Assume $P$ is a prime ideal such that $P_1, P_2$ are distinct minimal prime ideals of $R$ contained in $P$. Then $R=P_1+P_2\subseteq P$, which contradicts $P$ being a prime ideal. Thus $P$ contains a unique minimal prime ideal.

The equivalence of (1) and (2) in the case that $R$ has only finitely many minimal prime ideals is established in \cite[Proposition 4]{BKP-2011}.
\end{proof}

Observe that Lemma~\ref{lem:semiqBaer} allows us to decompose a large class of modules over such rings. Note that semilocal rings $R$ with the Jacobson radical of $R$ equal to the prime radical of $R$, left (or right) perfect rings and piecewise prime rings (\cite{BKP-2003} or \cite{BPR-2013}) are examples of rings $R$ such that $R/P(R)$ is quasi-Baer and $R$ has only finitely many minimal prime ideals.

\begin{theorem}\label{thm:semiqBaer}
Let $\mathcal{P}=\{P_i\}_{i\in I}$ be the set of minimal prime ideals of $R$, and let $K$ be a right $R$-module.
\begin{enumerate}
    \item Suppose that every prime ideal of $R$ contains a unique minimal prime ideal. Then $\underline{r}_R(mR)$ contains a nonempty finite intersection of elements of $\mathcal{P}$ for each $m\in M$ if and only if $M=\oplus_{i\in I}\underline{\ell}_M(P_i)$.
    \item Suppose that $R/P(R)$ is a quasi-Baer ring and $I=\{1,2,\dotsb,n\}$. Then:
        \begin{enumerate}
            \item $\bigcap\limits_{i=1}^nP_i^{m_i}\subseteq\underline{r}_R(M)$ for some $k_i\geq1$ if and only if $M=\underline{\ell}_M(P_1^{k_1})\oplus\dotsb\oplus\underline{\ell}_M(P_n^{k_n})$;
            \item if $M=K/KP(R)$ then $M=\underline{\ell}_M(P_1)\oplus\dotsb\oplus\underline{\ell}_M(P_n)$.
        \end{enumerate}
\end{enumerate}
\end{theorem}

\begin{proof}
The proof of (1) follows from Lemma~\ref{lem:semiqBaer} and Theorem~\ref{thm:infinite}, and (2) follows from Corollary~\ref{thm:main}.
\end{proof}

Note that for any ring $R$ with only finitely many minimal prime ideals, $R/P(R)$ has a right ring of quotients which is quasi-Baer (e.g., its quasi-Baer hull) with only finitely many minimal prime ideals \cite[Theorem 3.13]{BPR-2009}. Moreover, if $R$ is quasi-Baer and $P$ is a prime ideal, then either $P=eR$ for some $e=e^2\in R$ or $P_R\mathop{\leq^{\rm ess}\!}R_R$ (cf. \cite[Proposition 2.2]{BKP-2003}).

The following corollary is an application of Theorem~\ref{thm:semiqBaer}.

\begin{corollary}\label{cor:semiqBaer}
Let $T$ be an $n\times n$ generalized upper triangular matrix ring with $R_\alpha$ the ring in the $\alpha$-th diagonal entry, and $P_\alpha$ be the subset of $T$ with $0$ in the $\alpha$-th diagonal entry. Take $A$ to be the intersection of the $P_\alpha$, and let $K$ be a right $T$-module with $M=K/KA$. Then $M=\underline{\ell}_M(P_1)\oplus\dotsb\oplus\underline{\ell}_M(P_n)$. Moreover, if each $R_\alpha$ is a prime ring, then $T/P(T)$ is quasi Baer and $P_\alpha$ is a minimal prime ideal of $T$.
\end{corollary}

Note that if $R$ is a quasi Baer ring of $T$-dimension $n$, then $R$ is ring isomorphic to an $n\times n$ generalized triangular matrix ring $T$ with prime rings on the diagonal and with minimal prime ideals $P_\alpha$, as in Corollary~\ref{cor:semiqBaer} (c.f. \cite[Theorem 4.4]{BHKP-2000} or \cite[p.162]{BPR-2013}). Also each $P_\alpha=e_\alpha T$ for some left semicentral idempotent $e_\alpha\in T$ or $P_\alpha$ is essential in $T$. Thus either $\underline{\ell}_M(P_\alpha)=M(1-e_\alpha)$ or $\underline{\ell}_M(P_\alpha)$ is a submodule of $Z(M)$. In particular, note that any piecewise prime ring satisfies these conditions. Also observe that any right hereditary right noetherian ring is piecewise prime.

\section{Strongly p-Nilary Decompositions}\label{sec:nilary}

The main result of this section, Theorem~\ref{thm:nilarydecomposition}, extends the characterization of finitely generated torsion modules over Dedekind domains to a large class of noncommutative rings. From \cite{BKP-2013}, we use the following generalization of primary ideals from commutative ring theory and related concepts.

Let $I$ be an ideal of $R$. The \textit{pseudo-radical of} $I$, denoted $\sqrt{I}$, is defined as $\sqrt{I}=\sum\{V\trianglelefteq R\mid V^n\subseteq I$ for some $n\geq1\}$. $I$ is a \textit{strongly p-nilary ideal} of $R$ if and only if $\sqrt{I}$ is a prime ideal of $R$. Note that $\{0\}$ is a strongly p-nilary ideal of $R$ if and only if the sum of all nilpotent ideals is a prime ideal of $R$. Also, a set of strongly p-nilary ideals $Q_1,Q_2,\dots,Q_n$ of $R$ such that $I=Q_1\cap Q_2\cap\cdots\cap Q_n$ forms a \textit{minimal strongly p-nilary decomposition} of $I$ if and only if
\begin{inparaenum}
    \item for each $i$, $1\leq i\leq n$, $I\neq\bigcap\limits_{j\neq i}Q_j$, and
    \item for any subset $S\subseteq\{1,\dots,n\}$ with $\lvert S\rvert\geq2$, the ideal $\bigcap\limits_{s\in S}Q_s$ is not strongly p-nilary.
\end{inparaenum}

The following lemma can be found in \cite[Theorem 2.15]{BKP-2013}.

\begin{lemma}\label{lem:ACCnilary}
Suppose that $R$ has ACC on ideals. Then the following conditions are equivalent:
\begin{enumerate}
    \item For each pair of ideals $A,B\trianglelefteq R$, there exists a positive integer $k$ such that $A^k\cap B^k\subseteq AB$.
    \item Each $I\trianglelefteq R$ has a minimal strongly p-nilary decomposition.
\end{enumerate}
\end{lemma}

\begin{theorem}\label{thm:nilarydecomposition}
Suppose that $R$ has ACC on ideals, every pair of incomparable prime ideals of $R$ is comaximal, and for any ideals $A,B\trianglelefteq R$, $A^k\cap B^k\subseteq AB$ for some positive integer $k$. Let $M$ be a nonzero right $R$-module.
\begin{enumerate}
\item If $\underline{r}_R(M)\neq0$, then there exists a set $\{P_i\}_{i=1}^n$ of prime ideals and positive integers $k_i$, $1\leq i\leq n$, such that $\bigcap\limits_{i=1}^nP_i^{k_i}\subseteq\underline{r}_R(M)$, each $\underline{\ell}_M(P_i^{k_i})\neq0$, and $M=\oplus_{i=1}^n\underline{\ell}_M(P_i^{k_i})$.
\item Suppose that $\{Q_i\}_{i\in I}$ is the set of minimal prime ideals of $R$, with $\lvert I\rvert\geq2$, and $Y$ a set of generators of $M$. If $\underline{r}_R(yR)\neq0$ for each $y\in Y$, then $M=\oplus_{i\in I}C(Q_i)$.
\end{enumerate}
\end{theorem}

\begin{proof}
(1) The proof is similar to that of (2) and follows from Corollary~\ref{thm:main}.

(2) By Lemma~\ref{lem:ACCnilary}, there exists a minimal strongly p-nilary decomposition $\underline{r}_R(yR)=\bigcap\limits_{i=1}^nX_i$, where $\{X_i\}_{i=1}^n$ is a set of strongly p-nilary ideals of $R$. We have the existence of a nonzero prime ideal $P_i=\sqrt{X_i}$ and a positive integer $k_i$ such that $P_i^{k_i}\subseteq X_i\subseteq P_i$ for each $i$, since each $X_i$ is finitely generated. Note that $P_i$ contains a minimal prime ideal, say $Q_i$, and $Q_i^{k_i}\subseteq P_i^{k_i}$. Then $\bigcap\limits_{i=1}^nQ_i^{k_i}\subseteq\bigcap\limits_{i=1}^nX_i=\underline{r}_R(yR)$, Theorem~\ref{thm:infinite} yields $M=\oplus_{i\in I}^nC(Q_i)$. Because $\underline{r}_R(yR)=\bigcap\limits_{i=1}^nX_i$ is a minimal strongly p-nilary decomposition, it follows from Theorem~\ref{thm:minimal} that $C(Q_i)\neq0$ for each $i$.
\end{proof}

Note that the class of rings which satisfy the condition that every incomparable pair of prime ideals is comaximal is closed under taking direct products, matrices, and generalized triangular matrices. Also, this condition implies that distinct minimal prime ideals are pairwise comaximal, so that Theorem~\ref{thm:semiqBaer}(1) may be applicable.

\begin{examples}
Any finite direct sum of matrices over the following rings are examples of rings satisfying the hypothesis of Theorem~\ref{thm:nilarydecomposition}:
\begin{enumerate}
     \item any ring $R$ such that $R$ has ACC on ideals, every nonzero prime ideal of $R$ is maximal, and $R$ has the right AR-property for ideals (cf. \cite[pp. 190--193]{GW-2000}). In particular, any Dedekind domain;
    \item right duo rings with ACC on ideals such that every nonzero prime ideal is maximal (cf. \cite[Theorem 2.15 and Proposition 2.16]{BKP-2013}). For example, any generalized ZPI ring \cite[pp.469--477]{Gil-1972} in which minimal primes are pairwise comaximal (e.g., R=$\mathbb{Z}\oplus\mathbb{Z}_{p^n}$);
    \item local rings with nilpotent Jacobson radicals and ACC on ideals (cf. [7, Corollary 3.18]). In particular, any semisimple Artinian ring is a finite direct sum of such rings.
\end{enumerate}
\end{examples}

\section{Torsion Theory Induced by Pairwise Comaximal Ideals}\label{sec:torsion}

In this section, we develop a preradical $\gamma$ and its radical closure $\bar{\gamma}$ based on Theorem~\ref{thm:infinite}. Our main goal is to obtain a decomposition of a given module into a direct sum of a torsion module and a torsion-free module using the torsion theory that is developed. The torsion modules are defined to have the decomposition of Theorem~\ref{thm:infinite} or at least essentially contain such a decomposition.

For this section, we need basic terminology and facts of torsion theory. The definitions and results can be found in \cite[Chapter VI]{Sten-1975} or \cite[Chapters I, II]{BKN-1982}. We denote the category of all right $R$-modules by $\mathcal{M}_R$.

Given a nonempty set $\mathcal{X}$ of pairwise comaximal ideals of a ring, we define a preradical $\gamma_{\mathcal{X}}$ corresponding to $\mathcal{X}$, and list some basic properties.

\begin{definition}
Let $\mathcal{X}=\{X_i\}_{i\in I}$ be a fixed set of pairwise comaximal ideals of $R$. Define $\gamma_{\mathcal{X}}(M)=\bigl\{m\in M\mid \bigcap\limits_{i\in J}X_i^{k_i}\subseteq\underline{r}_R(mR)\text{ for some nonempty finite subset }J\subseteq I\text{ and positive integers }k_i\bigr\}$. We omit the subscript $\mathcal{X}$ when the context is clear.
\end{definition}

Note that $\gamma(M)$ is a submodule of $M$, and if $\mathcal{X}=\{X_i\}_{i=1}^n$ is finite, then $\gamma(M)=\sum\{N\leq M\mid\bigcap\limits_{i=1}^nX_i^{k_i}\subseteq\underline{r}_R(N)\text{ for some }k_i\geq1\}$.

\begin{proposition}\label{prop:radical}
Let $\{X_i\}_{i\in I}$ be a set of pairwise comaximal ideals of $R$. Then:
\begin{enumerate}
    \item $\gamma$ is a left exact preradical.
    \item $\gamma(M)=\oplus_{i\in I}C(X_i)$. In particular, $M$ is pretorsion-free if and only if $\underline{\ell}_M(X_i^{k_i})=0$ for each $k_i\geq1$.
    \item If either of the following conditions hold:
    \begin{enumerate}
        \item $\bigl(\bigcap\limits_{i\in J}X_i^{k_i}\bigr)^k$ is finitely generated for every nonempty finite subset $J\subseteq I$, and $k,k_i\geq1$, or
        \item $I$ is finite and for every $k_i\geq1$, $\bigl(\bigcap\limits_{i\in I}X_i^{k_i}\bigr)^k=\bigl(\bigcap\limits_{i\in I}X_i^{k_i}\bigr)^{k+1}$ for some $k\geq1$,
    \end{enumerate}
    then $\rho(M)=\bigl\{m\in M\,\Bigm|\,\bigl(\bigcap\limits_{i\in J}X_i^{k_i}\bigr)^n\subseteq\underline{r}_R(mR)\text{ for some }k_i,n\geq1$ and \\ nonempty finite subset $J\subseteq I\bigr\}$, is the smallest radical larger than $\gamma$ (i.e., $\rho=\bar{\gamma}$).
\end{enumerate}
\end{proposition}

\begin{proof}
(1) To show that $\gamma$ is a preradical let $f:M\longrightarrow N$ be an $R$-module homomorphism, and $m\in\gamma(M)$. Then $\underline{r}_R(mR)\supseteq\bigcap\limits_{i\in J}X_i^{k_i}$ for some nonempty finite set $J\subseteq I$ and $k_i\geq1$, so $\bigl(f(m)R\bigr)\bigl(\bigcap\limits_{i\in J}X_i^{k_i}\bigr)=f\bigl(m(\bigcap\limits_{i\in J}X_i^{k_i})\bigr)=f(0)=0$. Thus $f(m)\in\gamma(N)$. Therefore $\gamma$ is a preradical. The proof that $\gamma$ is left exact is straightforward.

(2) Since $\underline{\ell}_{\gamma(M)}(X_i)\subseteq\underline{\ell}_M(X_i)$ for all $i$, by Theorem~\ref{thm:infinite} we have $\gamma(M)\subseteq\oplus_{i\in I}C(X_i)$. Now let $m\in\oplus_{i\in I}C(X_i)$. Then there are a finite set $J\subseteq I$ and positive integers $k_j\geq1$ with $mR\subseteq\oplus_{j\in J}\underline{\ell}_M(X_j^{k_j})$. Hence $\bigcap\limits_{j\in J}X_j^{k_j}\subseteq\underline{r}_R\bigl(\oplus_{j\in J}\underline{\ell}_M(X_j^{k_j})\bigr)\subseteq\underline{r}_R(mR)$. So $m\in\gamma(M)$. Therefore $\gamma(M)=\oplus_{i\in I}C(X_i)$.

(3a) Observe that $\rho$ is a left exact preradical. We show that $\gamma(M/\rho(M))=0$. Suppose that $\bigl(kR+\rho(M)\bigr)/\rho(M)\leq\gamma\bigl(M/\rho(M)\bigr)$ for some $k\in M$. Then $\bigcap\limits_{i\in J}X_i^{k_i}\subseteq\underline{r}_R\bigl(k+\rho(M)\bigr)$ for some nonempty finite subset $J\subseteq I$, $k_i\geq1$. Since $\bigcap\limits_{i\in J}X_i^{k_i}$ is finitely generated, so is $k\bigcap\limits_{i\in J}X_i^{k_i}$. Note that $k\bigcap\limits_{i\in J}X_i^{k_i}\subseteq\rho(M)$ and $k\bigcap\limits_{i\in J}X_i^{k_i}$ finitely generated imply that there exist a nonempty finite subset $J'\subseteq I$ and $k'_j,m\geq1$ such that $k(\bigcap\limits_{i\in J}X_i^{k_i})(\bigcap\limits_{j\in J'}X_j^{k'_j})^m=0$. Then $k(\bigcap\limits_{i\in J\cup J'}X_i^{k_i})^m=0$, so $k\in\rho(M)$. Thus $kR\subseteq\rho(M)$. Therefore $\gamma\bigl(M/\rho(M)\bigr)=0$, so $\gamma(M)\subseteq\rho(M)$.

The method of proof that $\rho$ is a radical (i.e., that $\rho\bigl(M/\rho(M)\bigr)=0$) is similar to the argument above.

Suppose that $\tau$ is a radical containing $\gamma$. Consider $\rho\bigl(M/\tau(M)\bigr)$. Suppose that $nR/\tau(M)\leq M/\tau(M)$ such that there exist a nonempty finite subset $J\subseteq I$ and $k_i,m\geq1$ for which $nR(\bigcap\limits_{i\in J}X_i^{k_i})^m\subseteq \tau(M)$ and $nR(\bigcap\limits_{i\in J}X_i^{k_i})^{m-1}\nsubseteq\tau(M)$. Then $0\neq nR(\bigcap\limits_{i\in J}X_i^{k_i})^{m-1}/\tau(M)\subseteq\gamma\bigl(M/\tau(M)\bigr)=0$, which is a contradiction. Thus $\rho(M/\tau(M))=0$, which implies that $\rho(M)\subseteq \tau(M)$. Therefore $\rho$ is the smallest radical containing $\gamma$, so $\rho=\bar{\gamma}$.

(3b) The method of proof is similar to that of (3a).
\end{proof}

For example, if $R$ is a ring such that every maximal right ideal contains a maximal ideal (for example, if $R$ is a right quasi-duo ring \cite{Yu-1995}), and if $\mathcal{X}$ is the set of maximal ideals of $R$, then $\rm{Soc}(M)\subseteq\gamma(M)$. Thus, in this case, semisimple modules are $\gamma$-torsion modules.

Next we find conditions on $M$ and $R$ so that $M$ splits or essentially splits in $\gamma$ or $\bar{\gamma}$. Finding such conditions is of central importance in torsion theory.

We begin with the following technical lemma.

\begin{lemma}\label{lem:equalannihilators}
Let $M$ be an $R$-module such that $Z(M)=0$ and let $S$ and $K$ be submodules of $M$ such that $S\mathop{\leq^{\rm ess}\!}K$. If $\underline{\ell}_R\underline{r}_R(S)\subseteq\underline{r}_R\underline{r}_R(S)$, then $\underline{r}_R(S)=\underline{r}_R(K)$.
\end{lemma}

\begin{proof}
Since annihilation is order-reversing, $\underline{r}_R(K)\subseteq\underline{r}_R(S)$. Let $a\in\underline{r}_R(S)$, $k\in K$, and $L=\{x\in R\mid kax=0\}$. We show that $L_R\mathop{\leq^{\rm ess}\!}R_R$. Suppose that $t\in R-L$. Then $kat\neq0$. Since $S\mathop{\leq^{\rm ess}\!}K$, there exists $v\in R$ such that $0\neq katv\in S$. If $tv\underline{r}_R(S)=0$, then $tv\in\underline{\ell}_R\underline{r}_R(S)$. Hence $\underline{r}_R(S)tv=0$, so $katv=0$, a contradiction. So $tv\underline{r}_R(S)\neq0$. Then there exists $b\in\underline{r}_R(S)$ such that $tvb\neq0$. Then $katvb=0$. Thus $0\neq tvb\in L$, so $L\mathop{\leq^{\rm ess}\!}R$. Hence $kaL=0$ implies that $ka\in Z(M)=0$. Therefore $\underline{r}_R(S)=\underline{r}_R(K)$.
\end{proof}

From \cite{BMR-2002}, an \textit{FI-extending module} is a $R$-module $M$ such that every fully invariant submodule is essential in a direct summand of $M$. Observe that for an FI-extending module $\gamma(M)\mathop{\leq^{\rm ess}\!}\bar{\gamma}(M)\mathop{\leq^{\rm ess}\!}D$ where $D$ is a direct summand of $M$. Note that FI-extending modules are quite numerous since every finitely generated projective module over a semiprime ring has an FI-extending hull which, in general, is properly contained in its injective hull \cite[Theorem 6]{BPR-2009B}.

\begin{proposition}\label{prop:splitting}
Let $M$ be an FI-extending $R$-module such that $Z(M)=0$. If $R$ is commutative or semiprime, then $M=\gamma(M)\oplus F$ for some submodule $F\leq M$ such that $\gamma(F)=\bar{\gamma}(F)=0$.
\end{proposition}

\begin{proof}
Since $M$ is FI-extending, $\gamma(M)$ is essential in a direct summand, say $N$. Recall that in a semiprime ring $\underline{\ell}_R(I)=\underline{r}_R(I)$ for all $I\trianglelefteq R$. Hence $\underline{\ell}_R\underline{r}_R\bigl(\gamma(M)\bigr)=\underline{r}_R\underline{r}_R\bigl(\gamma(M)\bigr)$. Similarly, if $R$ is commutative, $\underline{\ell}_R\underline{r}_R\bigl(\gamma(M)\bigr)=\underline{r}_R\underline{r}_R\bigl(\gamma(M)\bigr)$. From Lemma~\ref{lem:equalannihilators}, $\underline{r}_R\bigl(\gamma(M)\bigr)=\underline{r}_R(N)$. By the definition of $\gamma(M)$, $N=\gamma(M)$, so $\gamma(M)$ is a direct summand of $M$. Since $\gamma(F)=F\cap\gamma(M)$, we have that $\gamma(F)=\bar{\gamma}(F)=0$ for any direct complement of $\gamma(M)$.
\end{proof}

Observe that in Proposition~\ref{prop:splitting}, $\gamma(M)=\bar{\gamma}(M)$, since $\gamma(M)$ is a closed submodule (i.e., $\gamma(M)$ has no nontrivial essential extension) of $M$.

To illustrate Proposition~\ref{prop:splitting}, our next result provides a large class of rings for which every projective module is nonsingular and FI-extending. Recall that an $AW^*$-algebra is a $C^*$-algebra which is a Baer ring \cite[Preface]{Kap-1968}. For example, any von Neumann algebra is an $AW^*$-algebra.

\begin{proposition}\label{prop:projective}
If $R$ is a right nonsingular semiprime quasi-Baer ring (e.g., right nonsingular prime rings, commutative Baer rings, and AW*-algebras), then every projective module is nonsingular and FI-extending.
\end{proposition}

\begin{proof}
From \cite[Theorem 4.7]{BMR-2002}, $R$ is right and left FI-extending. By \cite[Proposition 1.5 and Corollary 3.4]{BPR-2002}, every projective module is FI-extending. Clearly, every projective module is also nonsingular.
\end{proof}

Our next result is an application of Propositions~\ref{prop:splitting} and \ref{prop:projective} to operator theory.

\begin{corollary}\label{cor:operatortheory}
If $R$ is an AW*-algebra, then every finitely generated Hilbert C*-module $M$ is nonsingular and FI-extending (thus $\gamma(M)$ is a direct summand of $M$).
\end{corollary}

\begin{proof}
From \cite[Theorem 8.1.27 and p.352]{BL-2004}, every finitely generated Hilbert C*-module is projective. The remainder of the proof follows from Propositions~\ref{prop:splitting} and \ref{prop:projective}.
\end{proof}

Another condition which guarantees an FI-extending module $M$ splits in $\bar{\gamma}$ is that $\gamma$ be stable. So we look for conditions that ensure stability of $\gamma$. The following proposition gives conditions on $R$ that are sufficient for $\gamma$ to be stable.

\begin{proposition}\label{prop:stable}
Let $\{X_i\}_{i\in I}$ be a set of pairwise comaximal ideals of $R$. If, for each $L_R\mathop{\leq^{\rm ess}\!}R_R$, finite set $J\subseteq I$, and positive integers $k_j\geq1$, $L\bigl(\bigcap\limits_{j\in J}X_j^{k_j}\bigr)=\bigcap\limits_{j\in J}X_j^{k_j}$, then $\gamma$ is stable and $\gamma=\bar{\gamma}$.
\end{proposition}

\begin{proof}
Assume that $\gamma(M)=M$ and $w\in E(M)$. Then there exist a finite subset $J\subseteq I$ and positive integers $k_j\geq1$ so that $\bigcap\limits_{j\in J}X_j^{k_j}\subseteq\underline{r}_R(wR)$. Then $L=\{r\in R\mid wr\in M\}\mathop{\leq^{\rm ess}\!}R_R$. So $w\bigcap\limits_{i=1}^nX_i^{k_i}=wL\bigl(\bigcap\limits_{i=1}^nX_i^{k_i}\bigr)=0$. Thus $w\in\gamma\bigl(E(M)\bigr)$. Therefore $\gamma$ is stable. From \cite[pp.142, 152--153]{Sten-1975}, $\gamma=\bar{\gamma}$.
\end{proof}

Note that if $\rm{Soc}(R_R)\bigl(\bigcap\limits_{i=1}^nX_i\bigr)=\bigcap\limits_{i=1}^nX_i$ or if $\bigcap\limits_{i=1}^nX_i=\bigl(\bigcap\limits_{i=1}^nX_i\bigr)^2\subseteq\rm{Soc}(R_R)$ (e.g., if $\bigcap\limits_{i=1}^nX_i=0$), then $L\bigl(\bigcap\limits_{i=1}^nX_i\bigr)=\bigcap\limits_{i=1}^nX_i$ for all $L_R\mathop{\leq^{\rm ess}\!}R_R$.

\begin{proposition}\label{prop:stablesplitting}
If $M$ is an FI-extending $R$-module and the torsion theory associated with $\bar{\gamma}$ is stable, then $M=\bar{\gamma}(M)\oplus F$.
\end{proposition}

\begin{proof}
This result is a direct consequence of \cite[p.153, Proposition 7.2]{Sten-1975}.
\end{proof}

We conclude Section~\ref{sec:torsion} with some examples and an open question regarding the torsion theory developed.

\begin{examples}
\begin{enumerate}
    \item Let $R=T_n(A)$ be the ring of $n$-by-$n$ upper triangular matrices with entries in $A$, where $A$ is a ring with unity, and $\mathcal{X}=\left\{X_i\right\}_{i=1}^n$, where $X_i=\left\{(a_{ij})\in R\mid a_{ii}=0\right\}$ for each $i$, $1\leq i\leq n$. Note that $\mathcal{X}$ is a set of pairwise comaximal ideals in $R$. Also, $\bigl(\bigcap\limits_{i=1}^nX_i\bigr)^k=0$ for any $k\geq n$. Hence $\bar{\gamma}(M)=M$, so $\gamma(M)\mathop{\leq^{\rm ess}\!}M$ for any right $R$-module $M$.
    \item Let $R=\mathbb{Z}$, and $M=\mathbb{Z}_{p_1^\infty}\oplus\mathbb{Z}_{p_2^\infty}\oplus\cdots\oplus\mathbb{Z}_{p_n^\infty}\oplus\mathbb{Z}^{\omega}$, where each $p_i$ is a distinct prime and $\omega$ is any ordinal. Let $X_i=p_i\mathbb{Z}$ for each $i$, $1\leq i\leq n$. Then $\gamma(M)=\mathbb{Z}_{p_1^\infty}\oplus\mathbb{Z}_{p_2^\infty}\oplus\cdots\oplus\mathbb{Z}_{p_n^\infty}$. Note that $M=\bar{\gamma}(M)\oplus\mathbb{Z}^{\omega}$ and $\mathbb{Z}^{\omega}$ is $\bar{\gamma}$-torsion free.
\end{enumerate}
\end{examples}

Open Question: Characterize the radical closure of $\gamma$.

\bibliographystyle{amsplain}
\bibliography{References}

\providecommand{\bysame}{\leavevmode\hbox to3em{\hrulefill}\thinspace}
\providecommand{\MR}{\relax\ifhmode\unskip\space\fi MR }
% \MRhref is called by the amsart/book/proc definition of \MR.
\providecommand{\MRhref}[2]{%
  \href{http://www.ams.org/mathscinet-getitem?mr=#1}{#2}
}
\providecommand{\href}[2]{#2}
\begin{thebibliography}{10}

\bibitem{BKN-1982}
L.~Bican, T.~Kepka, and P.~N\v{e}mec, \emph{Rings, {M}odules, and
  {P}reradicals}, Lecture Notes in Pure and Applied Mathematics, vol.~75,
  Marcel Dekker, New York, 1982.

\bibitem{BHKP-2000}
G.~F. Birkenmeier, H.~E. Heatherly, J.~Y. Kim, and J.~K. Park, \emph{Triangular
  matrix representations}, J. Algebra \textbf{230} (2000), no.~2, 558--595.

\bibitem{BKP-2003}
G.~F. Birkenmeier, J.~Y. Kim, and J.~K. Park, \emph{Prime ideals of principally
  quasi-{B}aer rings}, Acta. Math. Hungar. \textbf{98} (2003), 217--225.

\bibitem{BKP-2011}
\bysame, \emph{The factor ring of a quasi-{B}aer ring by its prime radical}, J.
  Algebra Appl. \textbf{10} (2011), no.~1, 157--165.

\bibitem{BKP-2013}
\bysame, \emph{Right primary and nilary rings and ideals}, J. Algebra
  \textbf{378} (2013), 133--152.

\bibitem{BMR-2002}
G.~F. Birkenmeier, B.~J. Mueller, and S.~T. Rizvi, \emph{Modules in which every
  fully invariant submodule is essential in a direct summand}, Comm. Algebra
  \textbf{30} (2002), 1395--1415.

\bibitem{BPR-2002}
G.~F. Birkenmeier, J.~K. Park, and S.~T. Rizvi, \emph{Modules with fully
  invariant submodules essential in fully invariant summands}, Comm. Algebra
  \textbf{30} (2002), 1833--1852.

\bibitem{BPR-2009}
\bysame, \emph{Hulls of semiprime rings with applications to {C}*-algebras}, J.
  Algebra \textbf{322} (2009), 327--352.

\bibitem{BPR-2009B}
\bysame, \emph{Modules with {FI}-extending hulls}, Glasgow Math. J. \textbf{51}
  (2009), 347--357.

\bibitem{BPR-2013}
\bysame, \emph{Extensions of {R}ings and {M}odules}, Birkh\"{a}user/Springer,
  {N}ew {Y}ork, 2013.

\bibitem{BL-2004}
D.~P. Blecher and C.~Le~Merdy, \emph{Operator {A}lgebras and {T}heir {M}odules:
  {A}n {O}perator {S}pace {A}pproach}, Clarendon, Oxford, 2004.

\bibitem{Gil-1972}
R.~Gilmer, \emph{Multiplicative {I}deal {T}heory}, Pure and Applied
  Mathematics, no.~12, Marcel Dekker, New York, 1972.

\bibitem{GW-2000}
K.~R. Goodearl and R.~B. Warfield, \emph{An {I}ntroduction to {N}oncommutative
  {N}oetherian {R}ings}, Lond. Math. Soc. Student Texts, vol.~61, Cambridge
  University Press, Cambridge, 2000.

\bibitem{Hun-1974}
T.~W. Hungerford, \emph{Algebra}, Graduate Texts in Mathematics, vol.~73,
  Springer, New York, 1974.

\bibitem{Kap-1968}
I.~Kaplansky, \emph{Rings of {O}perators}, W. A. Benjamin, New York, 1968.

\bibitem{MR-2001}
J.~C. McConnell and J.~C. Robson, \emph{Noncommutative {N}oetherian {R}ings},
  revised ed., Graduate Studies in Mathematics, no.~30, American Mathematical
  Society, Providence, RI, 2001.

\bibitem{Row-1991}
L.~H. Rowen, \emph{Ring {T}heory}, student ed., Academic Press, San Diego,
  1991.

\bibitem{Sten-1975}
B.~Stenstr{\"o}m, \emph{Rings of {Q}uotients}, Springer-Verlag, New York, 1975.

\bibitem{Yu-1995}
H.~P. Yu, \emph{On quasi-duo rings}, Glasgow Math. J. \textbf{37} (1995),
  21--31.

\end{thebibliography}

\end{document}